\documentclass{amsart}
\usepackage{hyperref}
\usepackage{graphicx}
\usepackage{tikz}
\usepackage{enumerate}
\usepackage{latexsym,amsthm,amsfonts,amsmath,amssymb,amscd, mathrsfs}
\newcommand*{\mailto}[1]{\href{mailto:#1}{\nolinkurl{#1}}}

\newtheorem{theorem}{Theorem}[section]

\newtheorem{lemma}[theorem]{Lemma}

\newtheorem{remark}[theorem]{Remark}
\theoremstyle{definition}

\newcommand{\be}{\begin{equation}}
\newcommand{\ee}{\end{equation}}
\newcommand{\beq}{\begin{equation}}
\newcommand{\eeq}{\end{equation}}
\newcommand{\bea}{\begin{eqnarray}}
\newcommand{\eea}{\end{eqnarray}}

\newcommand{\noprint}[1]{}

\makeatletter
\newcommand{\dlmf}[1]{%
\cite[%
  \def\nextitem{\def\nextitem{, }}%
  \@for \el:=#1\do{\nextitem\href{http://dlmf.nist.gov/\el}{(\el)}}%
]{dlmf}%
}
\makeatother

\numberwithin{equation}{section}

\begin{document}

\title [The subspaces inclination]{Inclination of subspaces and decomposition of electromagnetic fields into potential and vortex components}


\author[M. Goncharenko]{ Maria Goncharenko}
\address{B. Verkin Institute for Low Temperature Physics and Engineering\\ 47, Nauky ave\\ 61103 Kharkiv\\ Ukraine}

\author[E. Khruslov]{ Evgen Khruslov}
\email{\href{mailto:khruslov@ilt.kharkov.ua}{khruslov@ilt.kharkov.ua}}


\begin{abstract} Using the notion of inclination of two subspaces $L$ and $M$ of Hilbert space $\mathcal{H}$, we prove the theorem on the extension of linear continuous functionals defined on the subspace $L$ to $\mathcal{H}$ so that the extended functionals vanish on the subspace $M$. We apply this theorem to study the question of decomposition of the electromagnetic field in resonator with ideally conducting boundary into potential and vortex components and derive the Korn-type inequality for vortex fields. \end{abstract}

\maketitle

\section*{Introduction}
\setcounter{section}{0}

The present study is motivated by the following question. Let $L$ and $M$ be two linear subspaces of the Hilbert space $\mathcal H$. What conditions must these subspaces satisfy so that any vector $x\in \mathcal H$ can be represented as a sum
\be\label{1} x=x^L+x^M, \quad \mbox{where} \ x^L\in L,\ x^M\in M,\ \ee
  and the inequality is fulfilled
\be\label{2}
\|x^L\|\leq A_1 \|x\|,\quad \|x^M\|\leq A_2  \|x\|,\ee
with constants $A_1$, $A_2$  not depended on $x$?

This is an abstract statement of a problem appearing in different  natural sciences.  In particular, in electrodynamics it is related to decomposition of an electromagnetic field in the domain with perfectly conducting boundary into the vortex and the potential components (\cite{Khruslov}). Note that for $L\cap M = \{ 0\}$ the decomposition \eqref{1} is a single-valued.

If the sum of subspaces $L$ and $M$ coincides with $\mathcal H$, equality \eqref{1} is evidently true,  but inequalities \eqref{2} is not obvious, and it requires additional information about $L$ and $M$  and estimates for constants $A_1$ and $A_2$.

That is why the question posed above involves, in the first place,  research of the closure of the sum $L+M$ of subspaces $L$ and $M$, and verification that its orthogonal complement is trivial: $(L+M)^\perp=0$.

The question of the closure of the sum of subspaces was studied earlier in a number of papers (\cite{Gurarii2} -- \cite{Kato}, \cite{Deutsch} -- \cite{Feshchenko}). In the papers \cite{Gurarii2} -- \cite{Kato} this question was studied for Banach spaces in a very general settings. The notion of inclination $\gamma_B(L,M) $ of two spaces $L$ and $M$ of a Banach space $B$ was introduced, and necessary and sufficient conditions for the closure of the sum of these subspaces were formulated in its terms. In the case of Hilbert spaces $\mathcal H=B$, the corresponding inclination $\gamma_{\mathcal H}(L,M) $ of the subspaces $L\subset \mathcal H$ and $M\subset \mathcal H$ is expressed in terms of the cosine of the angle $\varphi (L,M)$ between these subspaces by the formula $\gamma_{\mathcal H} (L,M) = \sqrt{1-\cos^2\varphi(L,M)}$. The definition of the angle $\varphi (L,M)$ between two subspaces of the Hilbert space was first given by K. Fridrichs in order to study the problem of the characteristic values of functions \cite{Fridrichs}. Subsequently, the notions of inclination and angle between subspaces were further developed and successfully applied in a number of branches of mathematics: the theory of bases, the theory of approximation and splines, operator theory (\cite{Gurarii2} -- \cite{Krein}).

In this paper, we use the definition of the inclination of subspaces $L, M \subset \mathcal H$, which takes into account the Hilbert structure of the space $\mathcal H$, that is, equivalent to the definition given, for example, in \cite{Deutsch}. In the first section, the notion of inclination $c(L,M)$ of subspaces $L$ and $M\subset \mathcal H$ is used to describe conditions on these subspaces under which representation \eqref{1} with estimates \eqref{2} holds. The main result is formulated in Theorem 1.1.

In the second section, the notion of inclination of subspaces is applied to the study of the possibility of extending linear continuous functionals $f\in L^*$ given on a subspace $L$ and vanishing on $L\cap M$ to functionals $\tilde f\in \mathcal H^*$ that vanish on a subspace $M\in\mathcal H$. It is proved that it is necessary and sufficient that the inclination of subspaces $L$ and $M$ be less than 1 (Theorem 1.2). As a corollary of this theorem, it is shown that conditions 1.2, formulated in Theorem 1.1 as sufficient for the validity of decomposition \eqref{1} -- \eqref{2}, are also necessary.

We note that the results obtained in Sections 1 and 2 are proved in this paper by quite elementary methods, although they may also be obtained as corollaries of profound results of previous papers (for example \cite{Deutsch}, \cite{Bottcher}, \cite{Ljance}). In the third Section, the results obtained in Sections 1, 2 are applied to the study of the decomposition of the electric component of the electromagnetic field in a domain with a perfectly conducting boundary into potential and vortex components. The inequality of Korn's type is derived for the vortex component of the field.

\section{The inclination of two subspaces of Hilbert space and decomposition of vectors into components from these subspaces}

Denote by $\mathcal H$ the Hilbert space (complex) with the scalar product $(u,v)$ and the norm $\Vert u\Vert = (u,u)^{1/2}$, $u, v\in\mathcal H$. Let $L$ and $M$ be two linear subspaces in it, $Q$ be their intersection (the case $Q=\{ 0\}$ is not excluded), and $L\ominus Q$ and $M\ominus Q$ be orthogonal complements to $Q$ in $L$ and $M$, respectively. We will assume that subspaces $L\ominus Q$ and $M\ominus Q$ are non-trivial. Taking into account the Hilbert structure of the space $\mathcal H$, we introduce the inclination of subspaces $L$ and $M$ by the formula
\be
\label{1.1}
c(L,M)=\underset{{{0\not= u\in L\ominus Q}\atop{0\not= v\in M\ominus Q}}}{\mathrm{supp}} \frac{\vert(u,v)\vert}{\Vert u\Vert\Vert v\Vert}.
\ee
Obviously $0\le c(L,M)\le 1$ and if $M\perp L$ or $(M\ominus Q)\perp (L\ominus Q)$, then $c(L,M)=0$. This is the reason for the name "inclination", since for $(M\ominus Q)\perp (L\ominus Q)$ we can say that the subspaces $L$ and $M$ are not inclined to each other. We exclude the cases $M\subseteq L$ and $L\subseteq M$ from consideration, since in these cases $c(L,M)$ is not defined by formula \eqref{1.1}, although it is natural to assume $c(L,M)=1$ in these cases. The definition of the inclination \eqref{1.1} actually coincides with the definition given in \cite{Deutsch} and expressed in terms of the angle $\varphi(L,M)$ between the subspaces $L$ and $M$ by the formula $c(L,M) =\cos\varphi(L,M)$. We will denote by $L+M=\{ x\in \mathcal H: x=x^L+x^M, x^L\in L, x^M\in M\}$ the sum of the subspaces $L$ and $M$ in $\mathcal H$.

\begin{lemma}
\label{lemma1.1}
If $c(L,M)<1$, then $L+M$ is a closed set in $\mathcal H$, and, hence, $L+M$ is a closed linear subspace of $\mathcal H$.
\end{lemma}

The statement of the Lemma follows from Theorem 1.3 in \cite{Deutsch} (see also \cite{Bottcher}, p. 1424, Example 3.2). But for convenience, we present it here by a simpler elementary method, which we will use in what follows.

\begin{proof}[Proof of Lemma \ref{lemma1.1}.]

Let $\{x_n\in L+M, n=1,2,...\}$ be a convergent sequence of vectors $x_n$ and
\be
\label{1.2}
x_n\to x\in L+M\subset {\mathcal H}\quad {\text{as}}\,\,\,n\to\infty.
\ee
Since $L+M=Q+(L\ominus Q)+(M\ominus Q)$, each vector $x_n\in L+M$ can be represented in a unique way in the form
\be
\label{1.3}
x_n=\hat y_n+\hat x_n^L+ \hat x_n^M.
\ee
Here $\hat y_n$ is the orthogonal projection $x_n$ onto the subspace $Q=L\cap M$, $\hat x_n^L\in L\ominus Q$, $\hat x_n^M \in M\ominus Q$. Therefore
\be
\label{1.4}
\Vert x_n\Vert^2=\Vert \hat y_n\Vert^2+\Vert x_n^L\Vert^2+\Vert x_n^M\Vert^2+2{\mathrm{Re}}(\hat x_n^L, \hat x_n^M).
\ee
Taking into account \eqref{1.1} we write the inequality
$$
\vert (\hat x_n^L, \hat x_n^M)\vert\le c(L,M)\Vert\hat x_n^L\Vert\Vert \hat x_n^M\Vert\le \frac{c(L,M)}{2} (\varepsilon\Vert \hat x_n^L\Vert^2+\varepsilon^{-1} \Vert \hat x_n^M\Vert^2),
$$
where $\varepsilon$ is any positive number. Then from \eqref{1.4} it follows that
\begin{multline*}
\Vert x_n\Vert^2 \ge \Vert \hat y_n\Vert^2+\Vert \hat x_n^L\Vert^2+\Vert \hat x_n^M\Vert^2-2\vert(\hat x_n^L, \hat x_n^M)\vert\\
\ge \Vert \hat y_n\Vert^2+\Vert \hat x_n^L\Vert^2 (1-\varepsilon c(L,M))+\Vert \hat x_n^M\Vert^2(1-\varepsilon^{-1} c(L,M)).
\end{multline*}
Hence, assuming that $0<c(L,M)<1$ and setting $\varepsilon =c(L,M)$ or $\varepsilon =c^{-1}(L,M)$, we obtain the inequalities
\be
\label{1.5}
\Vert \hat y_n\Vert\le \Vert x_n\Vert, \quad \Vert \hat x_n^L\Vert \le \frac{\Vert x_n\Vert}{\sqrt{1-c^2(L,M)}}, \quad \Vert \hat x_n^M\Vert \le \frac{\Vert x_n\Vert}{\sqrt{1-c^2(L,M)}}.
\ee
For $c(l,M)=0$ these inequalities are obvious, since the vectors $\hat y_n$, $\hat x_n^L$, $\hat x_n^M$ are mutually orthogonal. In a similar way, we estimate the differences
$$
\Vert \hat y_n -\hat y_m\Vert\le \Vert x_n - x_m\Vert,\quad \Vert \hat x_n^L- \hat x_m^L\Vert \le \frac{\Vert x_n -x_m\Vert}{\sqrt{1-c^2(L,M)}}, \quad \Vert \hat x_n^M- \hat x_m^M\Vert \le \frac{\Vert x_n -x_m\Vert}{\sqrt{1-c^2(L,M)}}.
$$
Since $\Vert x_n -x_m\Vert \to 0$ as $n,m\to\infty$, from these inequalities it follows that the sequences $\{\hat y_n\}_{n=1}^\infty$, $\{\hat x_n^L\}_{n=1}^\infty$, $\{\hat x_n^M\}_{n=1}^\infty$ are fundamental in the spaces $Q$, $L\ominus Q$, $M\ominus Q$, respectively. Since these spaces are complete, then
$$
\hat y_n \to \hat y\in Q, \quad \hat x_n^L\to \hat x^L\in L\ominus Q,\quad \hat x_n^M\to \hat x^M\in M\ominus Q,
$$
as $n\to\infty$ and, therefore,
$$
x_n\to \hat y +\hat x^L+\hat x^M\in L+M.
$$
By recalling \eqref{1.2}, we get $x=\hat y+\hat x^L+\hat y^M\in L+M$ for any $x\in\overline{L+M}$, and hence $L+M=\overline{L+M}$. The proof is complete.

\end{proof}

Let us now formulate the main result of this Section.
\setcounter{theorem}{0}
\begin{theorem}
\label{theorem1.1}

Let $L$ and $M$ be subspaces in $\mathcal{H}$ satisfying the conditions

1. $c=c(L,M)<1$;

2. $(L+M)^\perp =\{0\}$,
where $(L+M)^\perp$ is the orthogonal complement of $L+M$ in $\mathcal{H}$.

Then any vector $x\in\mathcal{H}$ can be represented as in \eqref{1} so that the inequalities \eqref{2} hold with constants
$$
A_1=a_1+\frac{1}{\sqrt{1-c^2}}, \quad A_2=a_2+\frac{1}{\sqrt{1-c^2}},
$$
where $0\le a_k\le 1$ ($k=1,2$) and $a_k=0$ if $L\cap M=\{0\}$.
\end{theorem}

\begin{proof}[Proof of Theorem \ref{theorem1.1}.]

By virtue of condition 1, according to Lemma \ref{lemma1.1} $L+M =\overline{L+M}$, that is $L+M$ is a closed linear subspace in $\mathcal{H}$, and according to condition 2, any vector from $\mathcal{H}$ orthogonal to $L+M$ is zero. Therefore, $L+M=\mathcal{H}$ and any vector $x\in\mathcal{H}$ can be represented in a form similar to \eqref{1.3}
\be
\label{1.6}
x=\hat y +\hat x^L+\hat x^M,
\ee
where $\hat y$ is the orthogonal projection of $x$ onto the subspace $Q=L\cap M$, $\hat x^L\in L\ominus Q$, $\hat x^M\in M\ominus Q$. Hence, similarly to \eqref{1.5}, we obtain the inequalities
\be
\label{1.7}
\Vert \hat y\Vert\le \Vert x\Vert, \quad \Vert \hat x^L\Vert\le \frac{\Vert x\Vert}{\sqrt{1-c^2}},\quad \Vert \hat x^M\Vert\le \frac{\Vert x\Vert}{\sqrt{1-c^2}}.
\ee
Now we represent the equality \eqref{1.6} as \eqref{1}
$$
x=x^L+x^M,
$$
where $x^L=\hat x^L +\hat a_1\hat y\in L$, $x^M=\hat x^M +\hat a_2\hat y\in M$, $\hat a_1$, $\hat a_2$ are arbitrary non-negative numbers such that $\hat a_1+\hat a_2=1$.

Taking into account \eqref{1.7} we get
$$
\Vert x^L\Vert\le \hat a_1 \Vert\hat y\Vert + \frac{1}{\sqrt{1-c^2}}\Vert x\Vert\le A_1 \Vert x\Vert,
$$
$$
\Vert x^M\Vert\le \hat a_2 \Vert\hat y\Vert + \frac{1}{\sqrt{1-c^2}}\Vert x\Vert\le A_2 \Vert x\Vert.
$$
Therefore, $\forall x\in \mathcal{H}$ can be represented in the form \eqref{1} with estimates \eqref{2} where
$$
A_k= a_k+\frac{1}{\sqrt{1-c^2}} \quad (k=1,2),
$$
$0\le a_k\le 1$, and $a_k=0$, if $\hat y =0$, and, in particular, if $L\cap M =\{0\}$.

\end{proof}

\setcounter{theorem}{0}

\begin{remark}
 Theorem 1.1 establishes the sufficiency of conditions 1, 2, which must be satisfied by subspaces $L, M\subset \mathcal{H}$ in order to realize $\forall x\in \mathcal{H}$ decomposition \eqref{1} with estimates \eqref{2}. It turns out that these conditions are also necessary. The necessity of condition 2 is evident, and the necessity of condition 1 follows from Theorem \ref{theorem2.1} of the next Section. Then Theorem \ref{theorem1.1} implies the necessity of condition 1 for the closedness of the sum $L+M$ (see Lemma \ref{lemma1.1}).
 \end{remark}
 
 \section{On the extension of linear continuous functionals}

Consider the set $F_Q\subset L^*$ of linear continuous functionals given on the subspace $L\subset \mathcal{H}$ and vanishing on the subspace $Q=L\cap M$ in $L$, i.e.
$$
F_Q =\{ f\in L^*: f(x)=0\,\,{\mathrm{as}}\,\,x\in Q\}.
$$

The following question is posed: is it possible to extend $f\in F_Q$ to a linear continuous functional $\tilde f\subset F_M \subset \mathcal{H}^*$, where $F_M =\{ \tilde f \in \mathcal{H}^*:\,\tilde f(x)=f(x)\,\,\mathrm{as}\,\,x\in L, \, f(x)=0\,\,\mathrm{as}\,\,x\in M,\, \Vert \tilde f\Vert_{\mathcal{H}^*}\le A\Vert f\Vert_L,\, A=A(N,M)<\infty\,\,\text{does not depend on}\,\, f\in F_Q\}$?

The answer is given by the following theorem.

\begin{theorem}
\label{theorem2.1}
In order to have an extension $\tilde f \in F_M \subset \mathcal{H}^*$ for the functional $f\in L^*$ (i.e., a mapping $F_Q$ to $F_M$) it is necessary and sufficient that the inclination $c(L,M)$ of the subspaces $L$ and $M$ satisfies the inequality $c(L,M) <1$. Moreover, if $\tilde f \in F_M$ is an extension of $f\in F_Q$, then
\be
\label{2.1}
\Vert \tilde f\Vert_{\mathcal{H}^*}\le \frac{1}{\sqrt{1-c^2(L,M)}}\Vert f\Vert_{L^*}.
\ee

\end{theorem}

\begin{proof}

Let $c(L,M)<1$. Then from Lemma \ref{lemma1.1} it follows that the sum of the subspaces $L$ and $M$ is a closed linear subspace of $\mathcal{H}$. For a given functional $f\in F_Q\subset L^*$, we define an extension $\hat f \in F_M\subset (L+M)^*$ by
$$
\hat f(x) =\begin{cases} f(x), & x\in L\\ 0, & x\in M\end{cases}
$$
i.e. for $x\in L+M$, $x=x^L+x^M=\hat y +\hat x^L +\hat x^M$ ($\hat y \in Q=L\cap M$, $\hat x^L \in L\ominus Q$, $\hat x^M \in M\ominus Q$) we assume that $\hat f(x) = f(\hat x^L)$. This definition is correct because $f(x)=0$ as $x\in Q$.

Let us estimate the norm of this functional in $L+M$. Using the inequality $\Vert x\Vert^2 \ge (1-c^2 (L,M))\Vert \hat x^L\Vert^2$ (see \eqref{1.7}), we get
\begin{multline}
\label{2.2}
\Vert \hat f\Vert_{(L+M)^*}=\underset{0\not= x\in L+M}{\mathrm{sup}} \frac{\vert \hat f(x)\vert}{\Vert x\Vert}=\underset{0\not= x^L\in L+M}{\mathrm{sup}} \frac{\vert  f(x^L)\vert}{\Vert x^L\Vert}\le \frac{1}{\sqrt{1-c^2}} \underset{x^L\in L\ominus Q}{\mathrm{sup}} \frac{\vert  f(x^L)\vert}{\Vert x^L\Vert}\\
\le \frac{1}{\sqrt{1-c^2}} \underset{0\not= x\in L}{\mathrm{sup}} \frac{\vert  f(x)\vert}{\Vert x\Vert}\le \frac{1}{\sqrt{1-c^2}} \Vert f\Vert_{L^*}.
\end{multline}
Now we extend the functional $\hat f\in (L+M)^*$ to the whole space $\mathcal{H}$ so that $\tilde f(x)=\hat f(x)$ as $x\in L+M$ and
\be
\label{2.3}
\Vert \tilde f\Vert_{\mathcal{H}^*}=\Vert \hat f\Vert_{(L+M)^*}.
\ee
For the Hilbert spaces $\mathcal{H}$ and $L+M\subseteq \mathcal{H}$, the possibility of such an extension follows from the Riesz theorem on the general form of the linear functional \cite{Dunford}.

Inequality \eqref{2.1} follows from \eqref{2.2} and \eqref{2.3}. Thus we have proved the sufficiency of the condition $c(L,M)<1$ for the mapping $F_Q\to F_M$. 

To verify its necessity, we use the method of proof by contradiction. Let us suppose that $c(L,M)=1$ and there exists a mapping $F_Q$ onto $F_M$ such that $f\to\tilde f$, $\Vert\tilde f\Vert_{\mathcal{H}^*}\le A\Vert f\Vert_{L^*}$, where $A$ does not depend on $f\in F_Q$. Since $c(L,M)=1$, from \eqref{1.1} it follows that there exist subsequences $\{u_n^\prime\}_{n=1}^\infty$ and $\{ v_n^\prime\}_{n=1}^\infty$ such that $u_n^\prime \in L\ominus Q$, $v_n^\prime \in M\ominus Q$, $\Vert u_n^\prime\Vert = \Vert v_n^\prime\Vert =1$, and $\vert( u_n^\prime, v_n^\prime)\vert\to 1$ as $n\to\infty$. We denote by $\varphi_n^\prime =\arg(u_n^\prime, v_n^\prime)$, i.e. $(u_n^\prime, v_n^\prime)=\vert (u_n^\prime, v_n^\prime)\vert e^{i\varphi_n^\prime}$ and set $u_n=u_n^\prime e^{-i\frac{\varphi_n^\prime}{2}}$, $v_n=v_n^\prime e^{i\frac{\varphi_n^\prime}{2}}$. Then $(u_n, v_n)= e^{-i\varphi_n^\prime}(u_n^\prime, v_n^\prime)=\vert (u_n^\prime, v_n^\prime)\vert$. Consequently there are sequences $u_n\in L\ominus Q$, $v_n \in M\ominus Q$ such that $\Vert u_n\Vert=\Vert v_n\Vert=1$, $\mathrm{Im}(u_n, v_n)=0$, $(u_n, v_n)\to 1$ as $n\to\infty$. Therefore
\be
\label{2.4}
\Vert u_n-v_n\Vert^2= \Vert u_n\Vert^2+\Vert v_n\Vert^2-2(u_n, v_n)\to 0 \,\,{\text{as}}\,\,n\to \infty.
\ee
  Let us introduce a sequence of functionals $\{f_n\in L^*\}$ assuming $f_n(x)=(x, u_n)$. Then $f_n(u_n)=1$ for $x\in L$, $\Vert f_n\Vert_{L^*}=1$ and $f_n(y)=0$ for $y\in Q=L\cap M$, i.e. $f_n\in F_Q$. By assumption, there are extensions $\tilde f_n \in F_M\subset \mathcal{H}^*$ of functionals $f_n$ to the whole space $\mathcal{H}$ such that $\tilde f_n (x)=f_n (x)$ for $x\in L$, $\tilde f_n (x)=0$ for $x\in M$ and
  \be
  \label{2.5}
  \Vert \tilde f_n \Vert_{\mathcal{H}^*}\le A\Vert f_n\Vert_{L^*}=A<\infty,
  \ee
  where constant $A$ does not depend on $n$.
  
Let us calculate the value of the functional $\tilde f_n\in F_M$ on the vector $w_n=u_n-v_n\in\mathcal{H}$. Since $u_n\in L\ominus Q$ and $v_n\in M\ominus Q$, $\tilde f_n (w_n)=\tilde f_n (u_n) -\tilde f_n (v_n)=1$. Therefore, taking into account \eqref{2.4}, we get
$$
\Vert \tilde f_n \Vert_{\mathcal{H}^*}= \underset{0\not= x\in \mathcal{H}}{\mathrm{sup}}\frac{\vert \tilde f_n(x)\vert}{\Vert x\Vert}\ge \frac{\vert \tilde f_n(w_n)\vert}{\Vert w_n\Vert}\to\infty,\quad n\to\infty.
$$
This contradicts \eqref{2.5} and, hence, $c(L,M)<1$.

\end{proof}

\setcounter{theorem}{0}
\begin{remark}
From Theorem \ref{theorem2.1} it follows that condition 1 in Theorem \ref{theorem1.1} is necessary. Indeed, if the decomposition \eqref{1} is valid for $\forall x\in\mathcal{H}$ with estimates \eqref{2}, i.e., $x=x^L+x^M$, $x^L\in L$, $x^M\in M$ and $\Vert x^L\Vert \le A_1\Vert x\Vert$, then any functional $f\in F_Q$ can be extended to a functional $\tilde f\in F_M$ by setting $\tilde f (x)=f(x^L)$. Then $\tilde f(x)=f(x)$ for $x\in L$, $f(x)=0$ for $x\in M$ and $\vert\tilde f(x)\vert = \vert f(x^L)\vert\le \Vert f\Vert_{L^*}\Vert x^L\Vert\le A_1 \Vert f\Vert_{L^*}\Vert x\Vert$ $\forall x\in \mathcal{H}$, i.e. $\Vert \tilde f\Vert_{\mathcal{H}^*} \le A_1\Vert f\Vert_{L^*}$ and, therefore, $\tilde f\in F_M$. Thus, Theorem \ref{theorem2.1} implies that $c(L,M)<1$.

\end{remark}

Let us give now the simplest example illustrating the notion of inclination and the results of Theorems \ref{theorem1.1} and \ref{theorem2.1}. Let $\mathcal{H}=l^2 (\mathbb{N})$, where the vectors $u=\{u_i\}_{i=1}^\infty$ from $l^2 (\mathbb{N})$ are real-valued (i.e. $u_i\in\mathbb{R}$) and the scalar product is defined by the formula
$$
(u,v)=\sum_{i=1}^\infty u_i v_i.
$$
Consider two subspaces in $\mathcal{H}=l^2 (\mathbb{N})$
\be
\label{2.6}
L=\{ v=(v_1,v_2,...)\in l^2 (\mathbb{N}): v_{2n-1}\in \mathbb{R}, v_{2n}=0\},
\ee
\begin{multline}
\label{2.7}
M=\{ w=(w_1,w_2,...)\in l^2 (\mathbb{N}): w_{2n-1}\in \mathbb{R}, w_{2n-1}=w_{2n}\theta_{2n}, \theta_{2n}\in \mathbb{R}, \\
0<\vert\theta_{2n}\vert <\infty\}.
\end{multline}
Vectors from $L$ and $M$ have the following structures $v=(v_1, 0, v_2, 0,...)$, $w=(w_1, w_1\theta_2, w_3, w_3\theta_4, ...)$. Thus, $Q=L\cap M=\{0\}$ and the inclination of these subspaces is defined by the equality
\be
\label{2.8}
c(L,M)= \underset{{{0\not= v\in L}\atop{0\not= w\in M}}}{\mathrm{sup}} \frac{\left \vert \sum\limits_{i=1}^\infty v_{2n-1}w_{2n-1}\right\vert}{\left ( \sum\limits_{n=1}^\infty v_{2n-1}^2\right )^{1/2}\left (\sum\limits_{n=1}^\infty w_{2n-1}^2 (1+\theta_{2n}^2)\right )^{1/2}} =\frac{1}{(1+\theta)^{1/2}},
\ee
where $\theta = \inf \theta_{2n}^2$.

Therefore $c(L,M) <1$ if $\theta >0$ and $c(L,M)=1$ if $\theta =0$.

For $\theta >0$ we have $L+M=\mathcal{H}$. Indeed, any vector $u=(u_1, u_2, ...) \in l^2 (\mathbb{N}$ can be represented in the form $u=v+w$, where $v_{2n-1}= u_{2n-1} -\frac{u_{2n}}{\theta_{2n}}$, $v_{2n}=0$, i.e. $v\in L$, and $w_{2n-1}=\frac{u_{2n}}{\theta_{2n}}$, $w_{2n}=u_{2n}$, i.e. $w\in M$.

Hence we obtain the following estimates for the vectors $v$ and $w$
$$
\Vert v\Vert \le \sqrt{\frac{\theta +1}{\theta}}\Vert u\Vert, \quad \Vert w\Vert \le \sqrt{\frac{\theta +1}{\theta}}\Vert u\Vert.
$$
Since, by virtue of \eqref{2.8},
\be
\label{2.9}
\frac{\theta+1}{\theta}=\frac{1}{1-c^2},
\ee
these estimates coincide with the estimates of the Theorem \ref{theorem1.1}.

Consider now the functional $f\in L^*$. By the Riesz theorem, there exists a vector $a\in L$, $a=(a_1,\theta, a_2, \theta,...)\in l^2 (\mathbb{N}$ such that 
\be
\label{2.10}
f(v)=(v,a)=\sum_{n=1}^\infty v_{2n-1} a_{2n-1}
\ee
for $v\in L$.

Extension $\tilde f\in F_M$ of the functional $f$ to the whole space $\mathcal{H}=l^2$ can be represented in the form 
\be
\label{2.11}
\tilde f(u)=(u,\tilde a)\quad \forall u\in\mathcal{H},
\ee
where $\tilde a= (\tilde a_1, \tilde a_2,...)\in l^2 (\mathbb{N}$, $\tilde a_{2n-1}= a_{2n-1}$, $\tilde a_{2n} = -a_{2n-1}\theta_{2n}^{-1}$, $n=1,2,...$

Indeed, according to \eqref{2.6}, \eqref{2.7}, \eqref{2.10}, \eqref{2.11} $\tilde f(v) = f(v)$ for $v\in L$; $\tilde f(w)=0$ for $w\in M$, and by virtue of \eqref{2.9}--\eqref{2.11} the estimate
$$
\Vert \tilde f\Vert_{\mathcal{H}^*}=\Vert \tilde a\Vert\le \sqrt{\frac{\theta +1}{\theta}}\Vert a\Vert =\frac{1}{\sqrt{1-c^2}}\Vert a\Vert=\frac{1}{\sqrt{1-c^2}} \Vert f\Vert_{L^*}
$$
holds and it coincides with the estimate of Theorem \ref{theorem2.1}.

\section{On decomposition of electric component of electromagnetic field in domain with perfectly conducting boundary into potential and vortex components}

Let $\Omega$ be bounded domain in $\mathbb{R}^3$ with smooth boundary $\partial\Omega =\Gamma$. The electric component of the electromagnetic field in a domain $\Omega$ with a perfectly conducting boundary $\Gamma$ is a vector field $u(x)\in L_2 (\Omega, \mathbb{C}^3)$ with a finite norm
\be
\label{3.1}
\Vert u\Vert_\Omega =\left \{ \int\limits_\Omega (\vert{\mathrm{rot}} u\vert^2 + \vert{\mathrm{div}} u\vert^2+ \vert u\vert^2)dx\right\}^{1/2}<\infty
\ee
that satisfies the following condition
\be
\label{3.2}
u_\tau (x)=0, \,\, x\in\Gamma.
\ee
Here and below we denote by $\vert\,\,\cdot\,\,\vert$ the norms of vectors from $\mathbb{C}^3$ (or $\mathbb{C}$), by $\langle \,\,\cdot\,\,\rangle$ the standard scalar product in $\mathbb{C}^3$ ($\mathbb{C}$), and by $u_\tau (x)$ and $u_\nu (x)$, respectively, the tangent and normal components of the field $u(x)$ on the surface $\Gamma$ at a point $x\in\Gamma$. For non-smooth vector functions satisfying condition \eqref{3.1}, these components are defined as elements of the space $H^{-1/2} (\Gamma)$ \cite{Duvant} and, therefore, the boundary condition \eqref{3.2} in the general case is understood in the generalized sense:
$$
u_\tau =0 \Leftrightarrow \int\limits_\Omega \langle u,\mathrm{rot} v\rangle dx =\int\limits_\Omega \langle \mathrm{rot}u, v\rangle dx, \,\,\,\forall v\in H^1 (\Omega, \mathbb{C}^3),
$$
or what is the same
$$
f_u (v)=\int\limits_\Gamma \langle v\wedge \nu, u \rangle d\Gamma =0\,\,\, \forall v\in H^{1/2} (\Gamma, \mathbb{C}^3),
$$
where $\nu = \nu(x)$ is the unit vector of the outer normal to the surface $\Gamma$ at the point $x\in\Gamma$, $\wedge$ is the vector product in $\mathbb{C}^3$, and since $u\in H^{-1/2}(\Gamma)$, the integral over $\Gamma$ is understood as a functional $f_u\in (H^{1/2}(\Gamma, \mathbb{C}^3))^*$. In paper \cite{Duvant}  was proved that the set of vector functions satisfying conditions \eqref{3.1}, \eqref{3.2} is a closed subspace $H^1_0 (\Omega, \mathbb{C}^3; \tau)$ of the Sobolev space $H^1 (\Omega, \mathbb{C}^3)$ of vector functions $v(x)$ satisfying the condition $v_\tau (x)=0$ on $\Gamma$. Hence, by virtue of the embedding theorem, $u_\tau \in H^{1/2}(\Gamma)$, $u_\nu \in H^{1/2}(\Gamma)$ and condition \eqref{3.2} can be understood in the usual sense.

Introducing in the subspace $H_0^1 (\Omega, \mathbb{C}^3; \tau)\subset H^1 (\Omega, \mathbb{C}^3)$ the scalar product $(\cdot ,\cdot)_\Omega$ compatible in the standard way $(u,u)^{1/2}_\Omega =\Vert u\Vert_\Omega$ with the norm \eqref{3.1}, we obtain a complete Hilbert space, which we denote by $\mathscr{H}$. Consider two linear subspaces in it
$$
\mathscr{L}=\{u\in\mathscr{H}:\, u(x)=0, \, x\in\Gamma\},
$$
$$
\mathscr{M}= \{ u\in \mathscr{H}: u=\nabla \varphi, \varphi \in H_0^1 (\Omega, \mathbb{C}), \Delta\varphi\in L_2 (\Omega)\}.
$$
It can be shown that these subspaces are closed in $\mathscr{H}$, and
$$
\mathscr{L} = \{u\in H_0^1 (\Omega, \mathbb{C}^3)\},
$$
$$
\mathscr{M}=\{u=\nabla \varphi, \varphi \in H_0^1 (\Omega, \mathbb{C})\cap H^2 (\Omega,\mathbb{C})\},
$$
where $H_0^1$, $H^2$ are the standard notation for Sobolev spaces $\stackrel{\circ}{W_2^1} =H_0^1$, $W_2^2=H^2$ (see, for example, \cite{Birman}). For this purpose, we use the well-known equality
$$
\Vert \nabla u\Vert_{L_2}^2=\Vert \mathrm{rot} u\Vert_{L_2}^2 + \Vert \mathrm{div} u\Vert_{L_2}^2, \,\, \forall u\in C_0^1(\Omega, \mathbb{C}^3),
$$
and inequality 
$$
\Vert\varphi\Vert_{H^2} \le C_2 (\varkappa) \Vert\Delta \varphi\Vert_{L_2}, \,\,\forall\varphi\in C_0^2 (\Omega),
$$
where $C_1 =\mathrm{const}$ ($\ge 1$), $C_2(\varkappa)$ is a constant depending on the curvature $\varkappa$ of the surface $\Gamma$ \cite{Ladyzhenskaya}.

Let us show now that for $\mathcal{H}=\mathscr{H}$, $L=\mathscr{L}$, $M=\mathscr{M}$ conditions 1 and 2 of Theorem 1.1 are satisfied, i.e. the inclination $c(\mathscr{L}, \mathscr{M})$ of subspaces $\mathscr{L}$ and $\mathscr{M}$ is less than $1$, and any vector $u\in\mathscr{H}$ orthogonal to the sum $\mathscr{L}+\mathscr{M}$ is zero. To verify that condition 1 of Theorem \ref{theorem1.1} is satisfied, we use Theorem \ref{theorem2.1}.

Let $f\in \mathscr{L}^*$ be a linear continuous functional defined on the space $\mathscr{L}$ and vanishing on $Q=\mathscr{L}\cap \mathscr{M}$, i.e. $f\in F_Q$. By virtue of the Riesz theorem, there exists a vector function $w\in \mathscr{L} =H_0^1(\Omega, \mathbb{C}^3)$ such that
\be
\label{3.3}
f(u)=(u,w)_\Omega =\int\limits_\Omega \left \{\langle \mathrm{rot}u, \mathrm{rot}w\rangle +\langle \mathrm{div}u, \mathrm{div}w\rangle +\langle u, w\rangle\right \} dx.
\ee
With the help of this equality, taking into account that $f(u)=0$ $\forall u\in Q=\mathscr{L}\cap \mathscr{M}=\{ u=\nabla\varphi: \varphi\in H_0^1 (\Omega), \nabla\varphi\in H_0^1(\Omega, \mathbb{C}^3)\}$, we conclude that $w(x)$ is a generalized solution of the following boundary value problem
 \be
 \label{3.4}
 \begin{cases}
 \mathrm{rot}\mathrm{rot}w(x)-\nabla \mathrm{div}w(x)+w(x)=j(x),\,\,x\in\Omega,\\
 w(x)=0,\,\,x\in\Gamma,
 \end{cases}
 \ee
 where $j(x)$ is a vector function from $H^{-1} (\Omega, \mathbb{C}^3)$ that satisfies the equation $\mathrm{div}j(x)=0$ in $\Omega$ in the sense of distributions:
$$
j(x)\in J(\Omega)=\{ j\in H^{-1} (\Omega, \mathbb{C}^3), \mathrm{div}j(x)=0\}.
$$
From this it follows that $w(x)$ satisfies the equation
\be
\label{3.5}
-\Delta \mathrm{div}w(x)+ \mathrm{div}w(x)=0,\,\,x\in\Omega,
\ee
and, thus, $\Vert \Delta \mathrm{div}w\Vert_{L_2(\Omega)}= \Vert \mathrm{div}w\Vert_{L_2(\Omega)}<\infty$.

Denote by $W(\Omega)$ the set of solutions of the boundary value problem \eqref{3.4} for all $j\in J(\Omega)$ and assume that
\be
\label{3.6}
S=\sup\limits_{W(\Omega)}\frac{\Vert \nabla\mathrm{div}w\Vert_{L_2(\Omega)}}{\Vert w\Vert_{H^1(\Omega)}}<\infty.
\ee
Such an inequality seems to be probable, since all $w(x)$ from $W(\Omega)$ belong to the space $H_0^1 (\Omega, \mathbb{C}^3)$ and satisfy equation \eqref{3.5} in the domain $\Omega$.

Let us show that, under this assumption, any functional $f\in F_Q\subset \mathscr{L}^*$  defined on the subspace $\mathscr{L}$ by formula \eqref{3.3} can be extended on the whole space $\mathscr{H}$  to a functional $\tilde f\in \mathscr{H}^*$  such that $\tilde f(u)=f(u)$ for $u\in \mathscr{L}$, $\tilde f(v)=0$ for $v\in \mathscr{M}$, and $\Vert \tilde f\Vert_{\mathscr{H}^*}\le \hat C(S) \Vert f\Vert_{\mathscr{L}^*}$, where $\hat C(S)$ does not depend on $f\in F_Q$.

We define the functional $\tilde f$ by formula
\be
\label{3.7}
\tilde f(u)= \int\limits_\Omega \left \{\langle \mathrm{rot}u, \mathrm{rot}w\rangle + \langle \mathrm{div}u, \mathrm{div}w\rangle + \langle u, w\rangle\right \} dx - \int\limits_\Gamma \langle u, \nu \mathrm{div}w\rangle d\Gamma, \,\,\forall u\in \mathscr{H},
\ee
where $w=w(x)$ is the same vector function as in the functional \eqref{3.3}, $\nu$ is the unit vector of the outward normal to the surface $\Gamma$. The surface integral in \eqref{3.7} is well defined, since $u\in H^1(\Omega, \mathbb{C}^3, \tau)$ and $\mathrm{div}w \in H^1(\Omega, \mathbb{C}^3)$ in view of assumption \eqref{3.6}. Taking this into account and using the embedding theorem for $H^1 (\Omega)$ in $L_2 (\Gamma)$, we obtain the inequality
$$
\left\vert \int\limits_\Gamma \langle u, \nu \mathrm{div}w\rangle d\Gamma \right\vert \le \Vert \mathrm{div}w\Vert_{L_2(\Gamma)}\Vert u_\nu\Vert_{L_2(\Gamma)} 
$$
$$
\le C(S+1)\Vert w\Vert_{H^1(\Omega)}\Vert u\Vert_{H^1(\Omega)}.
$$
Due to this inequality, from \eqref{3.7} it follows that
\be
\label{3.8}
\vert \tilde f (u)\vert \le (1+C(S+1))\Vert w\Vert_{H^1 (\Omega)}\Vert u\Vert_{H^1 (\Omega)}.
\ee
As proved in \cite{Birman}, there exists a continuous linear mapping $\mathscr{H}\to H_0^1 (\Omega, \mathbb{C}^3; \tau)$ and $\mathscr{L}\to H_0^1 (\Omega, \mathbb{C}^3)$ such that
$$
\Vert u\Vert_{H_0^1 (\Omega)} \le C_1 \Vert u\Vert_{\mathscr{H}}\,\,(1\le C_1 <\infty), \,\,\forall u\in \mathscr{H}
$$
and
$$
\Vert w\Vert_{H_0^1 (\Omega)} =\Vert w\Vert_{\mathscr{L}}, \,\, \forall w\in \mathscr{L}.
$$
Moreover, according to \eqref{3.3}
$$
\Vert w\Vert_{\mathscr{L}}=\Vert f\Vert_{\mathscr{L}^*}.
$$
Considering all this, with the help of \eqref{3.8}, we obtain
\be
\label{3.9}
\Vert\tilde f\Vert_{\mathscr{H}^*}=\sup\limits_{u\in \mathscr{H}}\frac{\vert \tilde f(u)\vert}{\Vert u\Vert_{\mathscr{H}}}\le C_1(1+C(S+1))\Vert f\Vert_{\mathscr{L}^*}= \hat C(S)\Vert f\Vert_{\mathscr{L}^*}
\ee
and, thus, the required inequality for $\tilde f$ is established.

Further, according to \eqref{3.7} and \eqref{3.3}, it is obvious that
\be
\label{3.10}
\tilde f(u)=f(u)\,\,{\mathrm{for}}\,\, u\in \mathscr{L}=H_0^1(\Omega,\mathbb{C}^3),
\ee
and for $u\in \mathscr{M}=\{ u=\nabla\varphi: \varphi\in H_0^1(\Omega, \mathbb{C})\cap H^2(\Omega,\mathbb{C})\}$ 
$$
\tilde f(u)=\int\limits_\Omega \left\{ \langle \Delta\varphi, \mathrm{div}w\rangle +\langle \nabla\varphi, w\rangle\right\} dx -\int\limits_\Gamma \langle \frac{\partial\varphi}{\partial\nu}, \mathrm{div}w\rangle d\Gamma =\int\limits_\Omega \langle \varphi, \Delta \mathrm{div}w-\mathrm{div}w\rangle dx
$$
and, thus, according to \eqref{3.5}
\be
\label{3.11}
\tilde f (u)=0\quad {\mathrm{for}}\,\,u\in \mathscr{M}.
\ee
Combining \eqref{3.9}-\eqref{3.10}, we conclude that any functional $f\in F_Q\subset \mathscr{L}^*$ ($Q=\mathscr{L}\cap \mathscr{M}$) can be extended to the functional $\tilde f\in F_{\mathscr{M}}\subset \mathscr{H}^*$. Therefore, according to Theorem \ref{theorem2.1}, the inclination of subspaces $\mathscr{L}$ and $\mathscr{M}$ is less than $1$, i.e., condition 1 of Theorem \ref{theorem1.1} is satisfied.

\begin{remark}
In the comparatively simple proof of this fact presented above, it was assumed that condition \eqref{3.1} is satisfied. Another proof that does not use this assumption is rather cumbersome and is not given here.
\end{remark}

Let us now show that condition 2 of Theorem \ref{theorem1.1} is also satisfied. Let $u\in (\mathscr{L}+\mathscr{M})^\perp\subset \mathscr{H}$. Then $(u,v)_\Omega=0$ $\forall v\in \mathscr{L}=H_0^1 (\Omega, \mathbb{C}^3)$ and $(u,v)_\Omega=0$ $\forall v\in \mathscr{M}=\{ v=\nabla\varphi, \varphi \in H_0^1 (\Omega)\cap H^2 (\Omega)\}$. Using these equalities and assuming that $u\in H_0^1 (\Omega, \mathbb{C}^3;\tau)\cap H^2 (\Omega, \mathbb{C}^3)$, we conclude that $u(x)$ is a solution of the following boundary value problem
\begin{eqnarray*}
\mathrm{rot}\mathrm{rot}u(x)-\nabla\mathrm{div}u(x)+u(x) =0,\,\, x\in\Omega,\cr
\mathrm{div}u(x) =0,\,\,x\in\Gamma,\cr
u_\tau (x) =0,\,\, x\in\Gamma.
\end{eqnarray*}
Hence it follows that
$$
\int\limits_\Omega \left \{ \vert \mathrm{rot}u\vert^2 +\vert \mathrm{div}u\vert^2 +\vert u\vert^2\right\}dx=0,
$$
and, therefore, $u\equiv 0$, i.e. condition 2 of Theorem \ref{theorem1.1} is satisfied.

According to Theorem \ref{theorem1.1}, any vector function from $\mathscr{H}$ can be represented as a sum of two vector functions from subspaces $\mathscr{M}$ and $\mathscr{L}$ with an estimate for the norms of the terms (see \eqref{1}, \eqref{2}). This representation is obviously not unique, if $\mathscr{M}\cap \mathscr{L}\not=\emptyset$.

Let us call vector functions from subspace $\mathscr{M}$ potential fields, and those from the subspace 
$$
\hat{\mathscr{L}}=\mathscr{L}\ominus(\mathscr{M}\cap \mathscr{L})
$$
vortex fields. It is clear from above that any vector function from $\mathscr{H}$  we can represent as a sum of two terms from subspaces $\mathscr{M}$  and $\mathscr{L}$, and such a decomposition is unique and estimates of the form \eqref{2} are valid. 

In conclusion, we show that the Korn-type inequality \cite{Horgan} holds for vortex fields:
\be
\label{3.12}
\Vert u\Vert^2_{H_0^1(\Omega)}\le \frac{1}{1-c^2(\hat {\mathscr{L}},{\mathscr{M}})}(\Vert \mathrm{rot}u\Vert^2_{L_2(\Omega)}+\Vert u\Vert^2_{L_2(\Omega)}).
\ee
First, we note that according to definition \eqref{1.1} and Theorem \ref{theorem2.1}, the inclinations $c(\hat{\mathscr{L}}, \mathscr{M})$ and $c(\mathscr{L}, \mathscr{M})$ of subspaces $(\hat{\mathscr{L}}, \mathscr{M})$ and $(\mathscr{L}, \mathscr{M})$ in $\mathscr{H}$ are equal and are less than $1$
$$
c=c(\hat{\mathscr{L}}, \mathscr{M})=c(\mathscr{L}, \mathscr{M})<1
$$
and 
\be
\label{3.13}
\vert (u,v)_\Omega\vert \le c\Vert u\Vert_{\mathscr{H}} \Vert v\Vert_{\mathscr{H}},
\ee
where $c=c(\hat{\mathscr{L}}, \mathscr{M})$, $\forall u\in \hat{\mathscr{L}}$ and $v\in \mathscr{M}$

Let us denote by $P_{\mathscr{M}}$ the orthogonal projection operator in $\mathscr{H}$ onto the subspace $\mathscr{M}$. Then for $u\in \hat{\mathscr{L}}$
$$
\sup\limits_{v\in \mathscr{M}}\frac{\vert (u,v)_\Omega\vert}{\Vert v\Vert_{\mathscr{H}}}=\Vert P_{\mathscr{M}}u\Vert_{\mathscr{H}},
$$
and, thus, according to \eqref{3.13},
\be
\label{3.14}
\Vert P_{\mathscr{M}}u\Vert^2_{\mathscr{H}}=\int\limits_\Omega \left \{ \vert \mathrm{div} P_{\mathscr{M}}u\vert^2 + \vert P_{\mathscr{M}}u\vert^2\right \}dx \le c^2 \Vert u\Vert^2_{\mathscr{H}}.
\ee
Let us represent $u\in \hat{\mathscr{L}}=\mathscr{L}\ominus (\mathscr{L}\cap \mathscr{M})$ in the form
\be
\label{3.15}
u= P_{\mathscr{M}}u+u^\prime,\quad u^\prime\in \mathscr{H}\ominus (\mathscr{L}\cap \mathscr{M}).
\ee
Evidently, $P_{\mathscr{M}}u^\prime =0$ and, hence, $\forall\varphi\in H_0^1 (\Omega)\cap H^2 (\Omega)$ since $\nabla\varphi\in \mathscr{M}$ the equality $(u^\prime, \nabla\varphi)_\Omega =0$ is true. Using this equality, we conclude that
\be
\label{3.16}
\mathrm{div}u^\prime (x) =0, \,\, x\in \Omega.
\ee
Taking into account \eqref{3.15}, \eqref{3.16}, we can rewrite inequality \eqref{3.14} in the form
$$
\int\limits_\Omega \vert \mathrm{div} u\vert^2 dx +\int\limits_\Omega \vert u-u^\prime\vert^2 dx\le \int\limits_\Omega \{ \vert\mathrm{rot}u\vert^2 + \vert\mathrm{div}u\vert^2 +\vert u\vert^2\} dx,
$$
whence it follows that
$$
(1-c^2)\int\limits_\Omega \vert\mathrm{div}u\vert^2 dx\le c^2 \int\limits_\Omega \vert\mathrm{rot}u\vert^2 dx + c^2 \int\limits_\Omega \vert u\vert^2 dx.
$$
Since $c<1$ we rewrite this inequality in th form
$$
\int\limits_\Omega \{ \vert\mathrm{rot}u\vert^2 + \vert\mathrm{div}u\vert^2 +\vert u\vert^2\} dx\le \frac{1}{1-c^2} (\Vert \mathrm{rot}u\Vert^2_{L_2(\Omega)}+ \Vert u\Vert^2_{L_2(\Omega)})
$$
and, recalling the well-known equality 
$$
\int\limits_\Omega \vert\nabla u\vert^2 dx = \int\limits_\Omega \{ \vert\mathrm{rot}u\vert^2 + \vert\mathrm{div}u\vert^2\}dx\,\,\,\forall u\in H_0^1 (\Omega, \mathbb{C}^3),
$$
we obtain the required inequality \eqref{3.12}.

  \end{document}